\documentclass[10pt,oneside,a4paper]{amsart}
\usepackage{latexsym}
\usepackage{amsfonts,amsmath,amssymb,indentfirst}
\usepackage[english]{babel}
\usepackage[all]{xy}
\usepackage[pdftex]{hyperref}
\newcommand{\bdism}{\begin{displaymath}}
\newcommand{\edism}{\end{displaymath}}
\newcommand{\cc}{\mathbb{C}}

\newcommand{\zz}{\mathbb{Z}}

\newcommand{\pp}{\mathbb{P}}

\newtheorem{main}{Theorem}

\newtheorem{theorem}{Theorem}[section]
\newtheorem{proposition}[theorem]{Proposition}

\newtheorem{lemma}[theorem]{Lemma}

\newtheorem{fact}[theorem]{Fact}

\address{Department of Mathematics, Duke University, Durham NC 27708-0320,
USA} \email{luca@math.duke.edu}

\author{Luca Fabrizio Di Cerbo}

\title{On the classification of toroidal compactifications with $3\overline{c}_{2}=\overline{c}^{2}_{1}$ and $\overline{c}_{2}=1$}
\begin{document}
\pagestyle{headings}
\begin{abstract}

We classify the smallest finite volume complex hyperbolic surfaces
with cusps which admit smooth toroidal compactifications and which are not birational to a bi-elliptic surface.
Remarkably, there is only one such surface which appears to be the
compactification of a Picard modular surface.
\end{abstract}
\maketitle

\tableofcontents

\section{Introduction}
\pagenumbering{arabic}

A fundamental problem in the theory of complex surfaces is the
classification of all surfaces of general type with
$3c_{2}=c^{2}_{1}$ and smallest possible Euler number. Note that the
Euler number of a surface of general type is positive and at least
three. This nontrivial fact follows from the standard theory of
surfaces of general type and the Bogomolov-Miyaoka-Yau inequality,
for more details see for example Chapter VII in \cite{Bar}.
Moreover, by Yau's solution of Calabi conjecture \cite{Yau} any such
surface is a quotient of the unit ball in $\cc^{2}$ by a torsion
free co-compact discrete subgroup of $\textrm{PU}(2,1)$.

Thus, the problem of finding all surfaces of general type with
\begin{align}\label{smallest}
3c_{2}=c^{2}_{1}, \quad c_{2}=3,
\end{align}
reduces to the classical and longstanding problem in geometric
topology of finding all the compact complex hyperbolic surfaces with
the smallest possible volume.

The first example of a surface as in \ref{smallest} was given by
Mumford in \cite{Mum}. Interestingly, the surface constructed by
Mumford has vanishing first Betti number. Thus, it has the same
Hodge diamond of the two dimensional complex projective space and it
is then an example of a ``fake projective plane''. In the same paper
Mumford raises the problem of enumerating all fake projective
planes. This difficult problem was solved only recently by Prasad
and Yeung in \cite{Pra} with the addition of Cartwright and Steger
\cite{Ste}. Note that an important ingredient in \cite{Pra} and
\cite{Ste} is the fact that the fundamental groups
$\Gamma\in\textrm{PU}(2,1)$ of fake projective planes are
necessarily \emph{arithmetic}, see \cite{Kli} and \cite{Yeu}. Recall
that nonarithmetic lattices are known to exist by the work of Mostow
and Deligne-Mostow, see \cite{Deligne}. For a survey of this
fascinating problem the interested reader may refer to the Bourbaki
report \cite{Rem}.

The natural question arises whether fake projective planes are the
only surfaces of general type with Euler number equal to three.
Surprisingly, this is not the case! In fact, Cartwright and Steger
in \cite{Ste} found an example of a surface as in \ref{smallest}
with first Betti number equal to two. Nevertheless, at the time of
writing, there is evidence that fake projective planes and the
Cartwright-Steger surface exhaust all surfaces of general type with
Euler number three, see the paper of Yeung \cite{Yeung}. In
conclusion, the classification of the smallest compact complex
hyperbolic surfaces seems to be completed.

The main goal of this paper is to initiate an analogous program for
cusped complex hyperbolic surfaces. More precisely, we classify the
smallest finite volume complex hyperbolic surfaces with cusps which
admit \emph{smooth} toroidal compactifications which are not birational to a bi-elliptic surface. 
Remarkably, there is only one such surface and it is associated to an arithmetic lattice.
The unique finite volume complex
hyperbolic surfaces we find appears to be a surface first
constructed by Hirzebruch in \cite{Hirzebruch} and later shown to be
the compactification of a Picard modular surface, see
\cite{Holzapfel}, \cite{Stover}. Thus, the group theoretical
implications of our main classification result are obtained
implicitly.

An outline of the paper follows. Section \ref{prelim} starts with a
short review of the geometry of smooth toroidal compactifications.
The problem of finding the smallest toroidal compactifications is
formulated in terms of \emph{logarithmic} Chern numbers. More
precisely, the smallest finite volume complex hyperbolic surfaces
with cusps which admit smooth toroidal compactifications correspond
to surfaces of logarithmic general type such that
\begin{align}\notag
3\overline{c}_{2}=\overline{c}^{2}_{1}, \quad \overline{c}_{2}=1,
\end{align}
where  by $\overline{c}_{i}$, $i=1, 2$, we denote the logarithmic
Chern numbers of the compactification. Finally, we state the main
results of this paper.

In Section \ref{different}, we show that a toroidal compactification
with $\overline{c}_{2}=1$ cannot be of Kodaira dimension equal to
two, one or minus infinity. In light of the Kodaira-Enriques
classification, we reduce the problem to the study of the case with
Kodaira dimension equal to zero.

In Section \ref{zero}, we study in detail the case with zero Kodaira
dimension. It is shown that the minimal model of a toroidal
compactification with $\overline{c}_{2}=1$, if it is not a bi-elliptic surface, then it must be  must be the product of
two elliptic curves with large automorphism group. Finally, an
explicit example is constructed and its uniqueness is proved.
Interestingly, the surface constructed already appears in a paper by
Hirzebruch \cite{Hirzebruch}. 

\subsection{Preliminaries and main results}\label{prelim}

The theory of compactifications of locally symmetric varieties has
been extensively studied, see for example \cite{Borel2}. Let
$\mathcal{H}^{n}$ be the $n$-dimensional complex hyperbolic space.
Finite volume complex hyperbolic manifolds then correspond to
torsion free non-uniform lattices in $\textrm{PU}(n,1)$. Let then
$\Gamma$ be any such lattice in $\textrm{PU}(n,1)$. It is well known
that when the parabolic elements in $\Gamma$ have no rotational
part, then the manifold $\mathcal{H}^{n}/\Gamma$ has a particularly
nice compactification $(X, D)$ consisting of a smooth projective
variety $X$ and an exceptional divisor $D$. More precisely, the
divisor $D$ is the union of smooth disjoint Abelian varieties with
negative normal bundle. The pair $(X, D)$ is referred as
\emph{toroidal} compactification of $\mathcal{H}^{n}/\Gamma$. For
more details about this construction the interested reader may refer
to \cite{Ash} and \cite{Mok}. Note that in \cite{Mok} this
construction is carried out without any arithmeticity assumption on
$\Gamma$.

Let us now describe in more details the two dimensional case. Let
$\mathcal{H}^{2}/\Gamma$ be a complex hyperbolic surface with cusps
which admit a smooth toroidal compactification $(X, D)$. It is well
known, see \cite{Sakai} and Section 4 in \cite{Luca1}, that $(X, D)$
is $D$-minimal of log-general type. We say that the pair $(X, D)$ is
$D$-minimal if $X$ does not contain any exceptional curve $E$ of the
first kind such that $D\cdot E\leq 1$. Moreover, by the
Hirzebruch-Mumford proportionality \cite{Mumford}, we have
\begin{align}\notag
3\overline{c}_{2}=\overline{c}^{2}_{1},
\end{align}
where $\overline{c}_{1}$ and $\overline{c}_{2}$ are the logarithmic
Chern numbers of the pair $(X, D)$. For the standard properties of
logarithmic Chern classes we refer to \cite{Kaw}. Recall that the
$\overline{c}^{2}_{1}$ of the pair $(X, D)$ is equal to the
self-intersection of the log-canonical divisor $K_{X}+D$, while
$\overline{c}_{2}$ is simply the topological Euler characteristic of
$X\backslash D$. Since $D$ consists of smooth disjoint elliptic
curves, we have
\begin{align}\notag
\overline{c}_{2}(X)=\chi(X)-\chi(D)=\chi(X)=c_{2}(X).
\end{align}
By construction $X\backslash D$ is equipped with a complete metric
with pinched negative sectional curvature. For this class of metrics
it is well known that the Pfaffian of the curvature matrix is
pointwise strictly positive, see Section 2 in \cite{Luca1}. Thus,
Gromov-Harder's generalization of Gauss-Bonnet \cite{Gro1} implies
that $\overline{c}_{2}$ has to be a strictly positive integer.

It is then interesting to look for toroidal compactifications with
the smallest possible value for the second logarithmic Chern number,
i.e., $\overline{c}_{2}=1$. Note how any such manifold will also
provide an example of a cusped complex hyperbolic manifold with the
smallest possible volume. The main purpose of this paper is to
provide a complete classification of smooth toroidal
compactifications with $3\overline{c}_{2}=\overline{c}^{2}_{1}$ and
$\overline{c}_{2}=1$ which are not birational to a bi-elliptic surface.

\begin{main}\label{primo}
There exists a unique toroidal compactification with
$3\overline{c}_{2}=\overline{c}^{2}_{1}$ and $\overline{c}_{2}=1$ which not birational to a bi-elliptic surface.
\end{main}

It is interesting to notice that because of the logarithmic
Bogomolov-Miyaoka-Yau inequality proved in \cite{TianY}, Theorem
\ref{primo} provides the classification of logarithmic pairs $(X,
D)$ of log-general type with
$3\overline{c}_{2}=\overline{c}^{2}_{1}$ and $\overline{c}_{2}=1$ for which $X$ is not birational to a bi-elliptic surface.

Here is an interesting corollary.

\begin{main}\label{secondo}
The lattice associated to a toroidal compactification with
$3\overline{c}_{2}=\overline{c}^{2}_{1}$ and $\overline{c}_{2}=1$ which is not birational to a bi-elliptic surfaces is
arithmetic.
\end{main}

In other words, any torsion free lattice $\Gamma\in
\textrm{PU}(2,1)$ whose parabolic elements have no rotational part
and with the smallest possible Euler characteristic must be
arithmetic if the associated smooth compactification is not birational to a bi-elliptic surface. The proof of Theorem \ref{secondo} is indirect. In fact,
it relies on the fact that the toroidal compactification identified
in Theorem \ref{primo} was already known to be associated to an
arithmetic lattice, see \cite{Holzapfel} and \cite{Stover}.
Nevertheless, the arithmeticity result stated in Theorem
\ref{secondo} seems not to be known.\\

For what regards the bi-elliptic case the classification is in progress. I can currently construct an example, moreover Matthew Stover has informed me that he can construct a different example. The classification of these bi-elliptic examples will be included in the next version of this work.

\vspace{0.5cm}

\noindent\textbf{Acknowledgements}. I would like to thank Professor
Mark Stern for his constant support and helpful discussions. Special
thanks go to Gabriele Di Cerbo for countless discussions and
encouragement. I also acknowledge useful discussions with Matthew
Stover. Finally, I would like to thank the organizers of the
conferences ``Algebraic \& Hyperbolic Geometry - New Connections''
and ``Geometria Algebrica nella Capitale'' where part of this work
was carried out.

\section{The case $\kappa\neq$ 0}\label{different}

In this section we show that a toroidal compactification with
$\overline{c}_{2}=1$ must have Kodaira dimension equal to zero. Let
us start by showing that $X$ cannot be of general type.

\begin{lemma}
Let $(X, D)$ be a toroidal compactification with
$\overline{c}_{2}=1$. Then $X$ cannot have $\kappa(X)=2$.
\end{lemma}

\begin{proof}

Recall that, given a surface $Y$ and letting $Bl_{k}(Y)$ be its blow
up at $k$ points, the second Chern number of $Bl_{k}(Y)$ is given by
\begin{align}\notag
c_{2}(Bl_{k}(Y))=k+c_{2}(Y).
\end{align}
Now, it is well known that the Euler characteristic of a minimal
surface of general type is strictly positive, see Chapter VII
\cite{Bar}. Since $c_{2}(X)=1$, we conclude that $X$ must be
minimal. Next, let us observe that
\begin{align}\notag
\overline{c}^{2}_{1}=c^{2}_{1}-D^{2}=3\overline{c}_{2}=3c_{2}
\end{align}
so that
\begin{align}\notag
0<c^{2}_{1}<3c_{2}
\end{align}
since $D^{2}<0$ and $c^{2}_{1}>0$ for any minimal surface of general
type. We then have $c^{2}_{1}\in\{1, 2\}$. But now for any complex
surface we must have
\begin{align}\notag
c^{2}_{1}+c_{2}=0\mod(12)
\end{align}
by Noether's formula, see page 9 in \cite{Friedman}. We therefore
conclude that $(X, D)$ cannot be such that $X$ is of general type.
\end{proof}

Let us proceed by ruling out the case of Kodaira dimension one.

\begin{lemma}
Let $(X, D)$ be a toroidal compactification with
$\overline{c}_{2}=1$. Then $X$ cannot have $\kappa(X)=1$.
\end{lemma}

\begin{proof}

Let us look for a moment at the minimal models of surfaces with
Kodaira dimension one, for details see Chapter VI in \cite{Bar}.
Thus given $X$, observe that there exists a unique minimal model $Y$
such that $c^{2}_{1}(Y)=0$ and $c_{2}(Y)\geq 0$. Now by Noether's
formula we have
\begin{align}\notag
c_{2}(Y)=12d
\end{align}
where $d\in\zz_{\geq0}$. Since the Chern number $c^{2}_{1}+c_{2}$ is
a birational invariant, we conclude that a surface with
$\kappa(X)=1$ must satisfy
\begin{align}\notag
c_{2}(X)=12d+k
\end{align}
with $d, k\in\zz_{\geq0}$. Therefore, if we want $c_{2}(X)=1$, we
must have $d=0$ and $k=1$. In other words, $X$ is the blow up at
just one point of a minimal elliptic surface $Y$ with zero Euler
number. For a minimal elliptic fibration
\begin{align}\notag
\pi: Y\longrightarrow E
\end{align}
with multiple fibers $F_{1}, ..., F_{k}$ of multiplicities $m_{1},
..., m_{k}$ we have
\begin{align}\notag
K_{Y}=\pi^{*}(K_{E}\otimes
L)\otimes\mathcal{O}_{Y}(\sum^{k}_{i=1}(m_{1}-1)F_{i})
\end{align}
where $L=(R^{1}\pi_{*}\mathcal{O}_{Y})^{-1}$, $d=deg(L)$ and
$c_{2}(Y)=12d$. In the case under consideration, we have $d=0$ and
then all the singular fibers of the elliptic fibration are multiple
fibers with smooth reduction, see Corollary 17 page 177 in
\cite{Friedman}. Consider now
\begin{align}\notag
f: X\longrightarrow E
\end{align}
where $f=\pi\circ Bl$ and $Bl: X\longrightarrow Y$ is the blow up
map. Let $D_{i}$ be an irreducible component of $D$. We then cannot
have $f(D_{i})=p_{i}$ for all $i$. If otherwise, there would exists
a smooth elliptic curve in $X\backslash D$. We conclude that the
image of at least one of the $D_{i}$'s under the elliptic fibration
is $E$. By the Hurwitz formula, the genus of $E$ must be $0$ or $1$.
Thus, if we want $\kappa(Y)=1$, we must assume the existence of
multiple fibers. Now, denote by $(Y, C)$ the blow down configuration
of $(X, D)$. Let us study the case $g(E)=1$ first. We then have that
an irreducible component of $C$, say $\Sigma$, is a holomorphic
$n$-section of the elliptic fibration. Moreover, $\Sigma$ is
normalized by a smooth elliptic curve, say $C'$, which is an
irreducible component of $D$. Let us denote by $Y'$ the fiber
product $Y\times_{E} C'$. Thus, $Y'\longrightarrow Y$ is a
$n$-covering map. Since $Y'$ has a holomorphic $1$-section, it
cannot have multiple fibers. We then have $\kappa(Y)=0$.

Let us now assume $g(E)=0$. In this case $L$ is trivial since
$deg(L)=0$. Again, there is a holomorphic $n$-section $\Sigma$ which
is normalized by a smooth elliptic curve $C'$. Following
\cite{Friedman} page 193, there is a finite cover $\pi':
Y'\longrightarrow C'$ of $\pi:Y\longrightarrow E$ with a
holomorphic $1$-section and such that $L'=\mathcal{O}_{C'}$.
We then have that $Y'=C'\times F$ for some elliptic curve $F$.
Because of Theorem 7.4 page 29 in \cite{Bar}, the Kodaira dimension
of $Y$ cannot be one.
\end{proof}

Let us conclude by showing that $X$ cannot be birational to a
rational or ruled surface.

\begin{lemma}
Let $(X, D)$ be a toroidal compactification with
$\overline{c}_{2}=1$. Then $X$ cannot have $\kappa(X)=-\infty$.
\end{lemma}

\begin{proof}

Recall that the minimal models of surfaces with negative Kodaira
dimension are $\pp_{2}$, the Hirzebruch surfaces $X_{e}$ and ruled
surfaces over Riemann surfaces of genus $g\geq 1$, see Chapter VI in
\cite{Bar}. Since
\begin{align}\notag
c_{2}(\pp_{2})=3,\quad c_{2}(X_{e})=4,
\end{align}
a toroidal compactification with $\overline{c}_{2}=1$ cannot have
these surfaces as minimal model. Moreover, since a negative elliptic
curve must occur as a $n$-section of the ruling, the Hurwitz formula
shows that $X$ has to be the blow up of a surface $Y$ ruled over an
elliptic curve. In this case $c_{2}(Y)=0$. Thus, $X$ is the blow up
of $Y$ at a single point. Since the rank of the Picard group of $X$
is three, by Proposition 3.8 in \cite{DiCerbo2} we have that $X$ can
at most have two cusps. Let us denote by $(Y, C)$ the blow down
configuration of $(X, D)$ and let us assume $C$ consists of one
irreducible components only. We then have that $C$ must be a
$n$-section of the ruling of $Y$. It is easily seen that $C$ cannot
be a $1$-section of the ruling. If otherwise, $X\backslash D$
contains a $\pp_{1}$ with just one puncture. For the same reason,
$C$ must be singular at a single point say $p$. Let $F$ be the fiber
of the ruling of $Y$ passing through the point $p$. Now, assume that
the tangent line of $F$ at $p$ does not coincide with any of the
tangent lines of $C$ at $p$. Blowing up the point $p$, we obtain
that the proper transform $\tilde{F}$ of $F$ in entirely contained
in $X\backslash D$. This is clearly impossible as
$\tilde{F}\simeq\pp_{1}$.

We proceed in a similar way if the tangent line of $F$ at $p$
coincides with one of the tangent lines of $C$ at $p$. In this case,
$\tilde{F}$ with one puncture is entirely contained in $X\backslash
D$. This is again impossible.

Let us conclude by studying the case when $(X, D)$ has two cusps. In
this situation $(Y, C)$ is such that $C$ consists of two irreducible
components, say $C_{1}$ and $C_{2}$, intersecting in a point $p$.
The irreducible components of $C$ might be singular at the point $p$
only, having an ordinary singular point there. Moreover, the
tangents lines of $C_{1}$ and $C_{2}$ at $p$ must be all distinct.
We can then proceed as in the one cusp case to get a contradiction.
\end{proof}

Let us summarize the results of this section into a proposition.

\begin{proposition}\label{reduction}
Let $(X, D)$ be a toroidal compactification with
$\overline{c}_{2}=1$. Then the Kodaira dimension of $X$ is zero.
\end{proposition}

Of course, it remains to be seen if any such example actually
exists. This problem is addressed in the next section.

\section{The case $\kappa$ = 0}\label{zero}

In light of Proposition \ref{reduction}, a toroidal compactification
with $\overline{c}_{2}=1$ must be birational to a minimal surface of
zero Kodaira dimension. Recall that minimal surfaces with zero
Kodaira dimension are given by:
\begin{itemize}

\item K3 surfaces, $c_{2}=24$;

\item Enriques surfaces, $c_{2}=12$;

\item Abelian surfaces, $c_{2}=0$;

\item bi-elliptic surfaces, $c_{2}=0$;

\end{itemize}
for details see again Chapter VI in \cite{Bar}. Thus, let $(X, D)$
be a as in Proposition \ref{reduction}. Since
$\overline{c}_{2}=c_{2}=1$, we have that $X$ is the blow up at
\emph{just} one point of an Abelian or bi-elliptic surface. Now, let
$D_{1}, ..., D_{k}$ be the irreducible components of the divisor
$D$. Since each $D_{i}$ is a smooth elliptic curve with negative
self-intersection, we have
\begin{align}\notag
\overline{c}^{2}_{1}(X)=(K_{X}+\sum_{i}D_{i})^{2}=K^{2}_{X}-\sum_{i}D^{2}_{i}=-1-\sum_{i}D^{2}_{i}.
\end{align}
But now $3\overline{c}_{2}(X)=\overline{c}^{2}_{1}(X)$, which
implies
\begin{align}\notag
-D^{2}_{1}-...-D^{2}_{k}=4.
\end{align}
Therefore, we have the following finite list of configurations:
\begin{itemize}

\item 1 cusp, $D^{2}_{1}=-4$;

\item 2 cusps, $D^{2}_{1}=-1$, $D^{2}_{2}=-3$ or $D^{2}_{1}=-2$, $D^{2}_{2}=-2$;

\item 3 cusps, $D^{2}_{1}=-1$, $D^{2}_{2}=-1$, $D^{2}_{3}=-2$;

\item 4 cusps, $D^{2}_{1}=D^{2}_{2}=D^{2}_{3}=D^{2}_{4}=-1$.

\end{itemize}

Now, let us denote by $(Y, C)$ the blow down configuration of $(X,
D)$. Since $Y$ is an Abelian or bi-elliptic surface, we have
$K_{Y}=0$. Thus, if $C_{i}$ is an irreducible component of $C$ in
$Y$, we have
\begin{align}\notag
p_{a}(C_{i})=1+\frac{C^{2}_{i}}{2}.
\end{align}
Note that $C^{2}_{i}\geq-2$. If $C^{2}=-2$, then $C_{i}$ is a smooth
rational curve. This is impossible since $Y$ is covered by
$\cc^{2}$. If $C^{2}_{i}=0$, with $C_{i}$ non-smooth, then $C_{i}$
is a rational curve with a single node or a cusp. This is again
impossible as, in both of these cases, $C_{i}$ is normalized by a
$\pp_{1}$. In conclusion, either $C_{i}$ is a smooth elliptic curve
with trivial self-intersection, or $C_{i}$ has a singular point, say
$p$, and $C^{2}_{i}=2n$ with $n\geq1$. Let us study the singular
case first. Thus, let
\begin{align}\notag
\pi:X\longrightarrow Y
\end{align}
be the blow up map at $p$. We then have
\begin{align}\notag
\pi^{*}C_{i}=D_{i}+rE
\end{align}
where $D_{i}$ is the proper transform of $C_{i}$ in $X$, $E$ is the
exceptional divisor and $r$ is the multiplicity of the singular
point $p$. Moreover, we have $D_{i}\cdot E=r$,
$D^{2}_{i}=C^{2}_{i}-r^{2}$ and
\begin{align}\label{genus}
2p_{a}(D_{i})-2=2p_{a}(C_{i})-2-r(r-1).
\end{align}
Now, if we want $D^{2}_{i}\leq-1$ with $C_{i}$ not smooth, we must
have
\begin{align}\notag
D^{2}_{i}=2n-r^{2}<-1.
\end{align}
Since $D_{i}$ is a smooth elliptic curve, the equation given in
\ref{genus} simplifies to the quadratic equation
\begin{align}\notag
r^{2}-r-2n=0,
\end{align}
whose solutions are given by
\begin{align}\notag
r_{1,2}=\frac{1\pm\sqrt{1+8n}}{2}.
\end{align}
Since $r$ is a positive integer, we just have to consider the plus
sign in the formula above. Thus, the self-intersection of $D_{i}$ is
given by
\begin{align}\notag
2n-(\frac{1+\sqrt{1+8n}}{2})^{2},
\end{align}
for $n\geq1$. This self-intersection is easily seen to be decreasing
in $n$ and for $n\geq7$ to be less than $-4$. All the possibilities
for $1\leq n\leq 6$ are then given by the following list:

\begin{align}\label{list}
& n=1, \quad C^{2}_{i}=2, \quad r=2;\\ \notag & n=3, \quad
C^{2}_{i}=6, \quad r=3;\\ \notag & n=6, \quad C^{2}_{i}=12,\quad
r=4. \notag
\end{align}






In conclusion, we then have to understand if on an Abelian or
bi-elliptic surface we can find curves with just one singular point
of multiplicities and self-intersections as in \ref{list}. Let us
start by studying the case when $Y$ is an Abelian surface. First, we
observe that the line bundle associated to a curve as in \ref{list}
must be ample.

\begin{lemma}\label{Nakai}
Let $C$ be an irreducible divisor on an Abelian surface $Y$ such
that $C^{2}>0$. Then $L=\mathcal{O}_{Y}(C)$ is ample.
\end{lemma}

\begin{proof}
Let $E$ be any curve on $Y$, we would like to show that $C\cdot
E>0$. Since $C^{2}>0$, we just need to study the curves $E\neq C$.
For these curves we clearly have $C\cdot E\geq0$. Assume then
$C\cdot E=0$. Let us denote by $t_{y}(E)$ the translation of the
curve $E$ by an element $y\in Y$. By appropriately choosing $y\in
Y$, we can assume that $t_{y}(E)\cap C\neq\{0\}$. Since the curve
$t_{y}(E)$ is numerically equivalent to $E$ we have then reached a
contradiction. We therefore conclude that $L$ is a strictly nef line
bundle with positive self-intersection. The lemma is now a
consequence of Nakai's criterion for ampleness of divisors on
surfaces, see Corollary 6.4 page 161 in \cite{Bar}.
\end{proof}

Next, we show that curves as in \ref{list} cannot exist on an
Abelian surface. The proof of this fact uses standard properties of
theta functions. Recall that any effective divisor on a complex
torus is the divisor of a theta function, see Th\'eor\`eme 3.1 page
43 in \cite{Deb}. Now, let $C$ be a reduced divisor as in Lemma
\ref{Nakai}. Then, if we let $V=\cc^{2}$ and $\pi: V\longrightarrow
V/\Gamma$ be the universal covering map, we have that
\begin{align}\label{theta}
\pi^{*}C=(\theta)
\end{align}
for some theta function on $V$. More precisely, we can find a
Hermitian form $H$, a character $\alpha: \Gamma\longrightarrow U(1)$
and a theta function satisfying \ref{theta} and the following
``normalized'' functional equation
\begin{align}\label{functional}
\theta(z+\gamma)=\alpha(\gamma)e^{\pi H(\gamma,
z)+\frac{\pi}{2}H(\gamma, \gamma)}\theta(z)=e_{\gamma}(z)\theta(z)
\end{align}
for any $z\in V$ and $\gamma\in \Gamma$. Note how $e_{\gamma}$ is
the factor of holomorphy for the line bundle $L=\mathcal{O}_{Y}(C)$.
There is then an identification between the space of sections of L
and the vector space of theta functions of type $(H, \alpha)$ on
$V$.

In light of the list obtained in \ref{list}, we are interested in
the case when $C$ has a singular point only. Thus, let $C\in|L|$ be
reduced divisor and let us denote by $C^{*}=C\backslash\{p\}$ the
smooth part of $C$. For every $q\in C^{*}$, $T_{q}C$ is a well
defined $1$-dimensional subspace of $T_{q}Y$. Thus, if we let
$z_{1}, z_{2}$ be coordinate functions for $V$, the equation for
$T_{q}C$ is given by
\begin{align}\notag
\sum^{2}_{i=1}\partial_{z_{i}}\theta(q)(z_{i}-q_{i})=0.
\end{align}
We can then consider a Gauss type map
\begin{align}\notag
G: C^{*}\longrightarrow \pp_{1}
\end{align}
where
\begin{align}\notag
G(q)=(\partial_{z_{1}}\theta(q): \partial_{z_{2}}\theta(q)).
\end{align}
We claim that since $C$ is reduced and $L$ is ample, then the Gauss
map cannot be constant. Let us proceed by contradiction. Say that
the image of the Gauss map is the point $[x_{1}: x_{2}]\in \pp_{1}$.
If $x_{2}\neq 0$, let us define the derivation
\begin{align}\notag
\partial_{w}:=\partial_{z_{1}}-k\partial_{z_{2}}
\end{align}
where $k=x_{1}/x_{2}$. If $x_{2}=0$, let us simply consider the
derivative along the second coordinate function, in other words
$\partial_{w}=\partial_{z_{2}}$. By construction, we have
$\partial_{w}\theta=0$ for all $q\in C^{*}$. Since $C$ is reduced,
the function
\begin{align}\notag
f=\partial_{w}\theta/\theta
\end{align}
is holomorphic on $V$ except at the singular points of $\pi^{*}C$.
By the Hartogs extension theorem, we know that $f$ can be extended
to a holomorphic function on $V$. Because of the functional equation
\ref{functional}, we have
\begin{align}\notag
f(z+\gamma)-f(z)=\pi H(\gamma, v)
\end{align}
for any $\gamma\in\Gamma$, where \begin{equation}\notag
v=\partial_{w}\left(
\begin{array}{ccc}
z_1 \\
z_2
\end{array}\right).
\end{equation}
This implies
\begin{align}\notag
f(z)=\pi H(z, v)+K
\end{align}
for some constant $K$. Since $f$ is holomorphic and $H$ is
anti-holomorphic in $z$, we have therefore reached a contradiction.
To sum up, we have shown that for any derivation $\partial_{w}$, the
function $\partial_{w}\theta$ cannot be identically zero on $C^{*}$.

Now, because of functional equation given in \ref{theta}, the
restriction of $\partial_{w}\theta$ to $\pi^{*}C$ can be considered
as a section of the line bundle $L$ restricted to $C$. Thus, the
intersection number $(\partial_{w}\theta)\cdot C$ coincides with the
self-intersection $C^{2}$. Let us now consider a derivation
$\partial_{w}$ with parameter $w$ determined by a generic point in
the image of the Gauss map. Say that the multiplicity of the
singular point $p$ is $r_{p}$. The intersection number of
$(\partial_{w}\theta)$ and $C$ at the singular point $p$ is then
computed by $r_{p}(r_{p}-1)$. Moreover, by construction
$\partial_{w}\theta$ vanishes somewhere on $C^{*}$. We conclude that
\begin{align}\label{obstruction}
C^{2}\geq r_{p}(r_{p}-1)+1.
\end{align}

Now in all of the cases given in \ref{list}, we have
\begin{align}\notag
C^{2}_{i}=r(r-1)
\end{align}
so that, using \ref{obstruction}, we can rule out the cases of one,
two and three cusps.

Let us summarize these results into a lemma.

\begin{lemma}\label{abelian}
Let $(X, D)$ be a toroidal compactification with
$\overline{c}_{2}=1$ and $\kappa(X)=0$. If $X$ is not bi-rational to a bi-elliptic surface, then $X$ is the blow up of
an Abelian surface. Moreover, $D$ consits of four disjoint smooth elliptic curves with $D^{2}_{1}=D^{2}_{2}=D^{2}_{3}=D^{2}_{4}=-1$. 
\end{lemma}

Thus, in light of  Lemma \ref{abelian}, we have
to classify the pairs $(Y, C)$ where $Y$ is an Abelian surface and
$C$ consists of four smooth elliptic curves intersecting in just one
point. We will show that, up to isomorphism, there is only one such
pair. This result follows from few geometrical facts.

\begin{fact}\label{fatto1}
Let $Y=\cc^{2}/\Gamma$ be an Abelian surface containing two smooth
elliptic curves $C_{1}$, $C_{2}$ such that $C_{1}\cdot C_{2}=1$.
Then $Y$ is isomorphic to the product $C_{1}\times C_{2}$.
\end{fact}
\begin{proof}
By translating the curves $C_{1}$ and $C_{2}$, we can always assume
that $C_{1}\cap C_{2}= \{(0, 0)\}$. The curves $C_{i}$, $i=1, 2$,
are then subgroups of $Y$. Thus, we can define the map
\begin{align}\notag
\varphi: C_{1}\times C_{2}\longrightarrow Y
\end{align}
which sends the point $(p, q)\in C_{1}\times C_{2}$ to $p-q\in Y$.
The map $\varphi$ is clearly one-to-one.
\end{proof}

\begin{fact}\label{fatto2}
Let $Y=\cc^{2}/\Gamma$ be an Abelian surface containing three smooth
elliptic curves $C_{i}$, $i=1, 2, 3$, such that $C_{1}\cap C_{2}\cap
C_{3}=\{(0, 0)\}$ and such that $C_{i}\cdot C_{j}=1$ for any $i\neq
j$. Then $Y$ is isomorphic to the product $C\times C$ where
$C_{i}=C$ for any $i$.
\end{fact}
\begin{proof}
By Fact \ref{fatto1}, we have that $Y=C_{1}\times C_{2}$. Since
$C_{3}\cdot C_{1}=1$, for $i=1, 2$, we conclude that $C_{3}=C_{i}$
for $i=1, 2$. The proof is complete.
\end{proof}

\begin{fact}\label{fatto3}
Let $Y$ be an Abelian surface which is the product of two identical
elliptic curves, say $C=\cc/\Lambda$. Let $(w, z)$ be the natural
product coordinates on $Y$. Then any smooth elliptic curve in $Y$,
passing through the point $(0, 0)$, is given by an equation of the
form $w=\alpha z$, with $\alpha$ such that
$\alpha\Lambda\subseteq\Lambda$.
\end{fact}
\begin{proof}
A subgroup in $\cc^{2}$ is given by an equation of the form
$w=\alpha z$. Finally, this equation makes sense on $Y$ if
$\alpha\Lambda\subseteq\Lambda$.
\end{proof}

\begin{fact}\label{fatto4}
Let us denote by $C_{\alpha}$ the curve in $Y=C\times C$ given by
the equation $w=\alpha z$ with $\alpha\Lambda\subseteq\Lambda$ and
$\alpha\neq 0$. Then $C_{0}\cdot C_{\alpha}=1$ if and only if
$\alpha\Lambda=\Lambda$.
\end{fact}
\begin{proof}
The intersection $C_{0}\cap C_{\alpha}$ consists of $[\alpha\Lambda:
\Lambda]$ distinct points, where by $[\alpha\Lambda: \Lambda]$ we
denote the index of the subgroup $\alpha\Lambda$ in $\Lambda$.
\end{proof}

Let us now go back to our original problem.  We want to classify all
the configurations of four elliptic curves $C_{i}$, $i=1, 2, 3, 4$,
in an Abelian surface $Y$ such that
\begin{align}\notag
C_{1}\cap C_{2}\cap C_{3}\cap C_{4}=\{p\}
\end{align}
for a point $p\in Y$ and
\begin{align}\notag
C_{i}\cdot C_{j}=1
\end{align}
for any $i\neq j\in\{1, 2, 3, 4\}$. Any such configuration will be
referred as \emph{good configuration}. Now, by translating the
$C_{i}'$s, we can assume the point $p$ to coincide with the origin
in $Y$. Because of Facts \ref{fatto1} and \ref{fatto2}, we can
assume $Y=C\times C$ with the curves $C_{1}$ and $C_{2}$ being the
factors in the splitting of $Y$. Because of Facts \ref{fatto3} and
\ref{fatto4}, we have to look for values of $\alpha$, say
$\alpha_{1}$ and $\alpha_{2}$, such that
\begin{align}\notag
C_{3}=C_{\alpha_{1}}, \quad C_{4}=C_{\alpha_{2}}.
\end{align}
Now, for a generic elliptic curve $C=\cc/\Lambda$, the only values
of $\alpha$ such that $\alpha\Lambda=\Lambda$ are given by
$\alpha=\pm 1$. If this is the case, note that $C_{1}\cap C_{-1}$
consists of four disjoint points. These points are exactly the
two-torsion points of the lattice $\Lambda$. In conclusion, for a
generic elliptic curve $C$, the Abelian surface $Y=C\times C$ cannot
support a good configuration. It remains to treat the case of a
non-generic elliptic curve $C$. Recall that there are only two
elliptic curves with non-generic automorphism group. These elliptic
curves correspond to the lattices $\Lambda_{(1, i)}=\zz+\zz i$,
$\Lambda_{(1, \tau)}=\zz+\zz\tau$ where $\tau=e^{\frac{\pi i}{3}}$.

For the lattice $\Lambda_{(1, i)}$, we have four choices of the
value of $\alpha$ so that $\alpha\Lambda_{(1, i)}=\Lambda_{(1, i)}$:
\begin{align}\notag
\alpha=1, i, i^{2}, i^{3}.
\end{align}
It turns out that none of the possible choices involving these
parameters give a good configuration. To this aim, it suffices to
observe that the configuration
\begin{align}\notag
w=0,\quad z=0,\quad w=z, \quad w=iz,
\end{align}
is such that
\begin{align}\notag
C_{1}\cap C_{i}=\{(0, 0), \quad (1/2+i/2, 1/2+i/2)\}.
\end{align}
Any other configuration is either isomorphic to the one above or
fails to be a good configuration by completely analogous reasons.

For the lattice $\Lambda_{(1, \tau)}$, we have six choices of the
value of $\alpha$ so that $\alpha\Lambda_{(1, \tau)}=\Lambda_{(1,
\tau)}$:
\begin{align}\notag
\alpha=1, \tau, \tau^{2}, \tau^{3}, \tau^{4}, \tau^{5}.
\end{align}
Let us observe that
\begin{align}\notag
w=0,\quad z=0,\quad w=z, \quad w=\tau z,
\end{align}
is a good configuration. In fact, the curves $C_{1}$ and $C_{\tau}$
intersect at the points whose z values satisfy the equality
\begin{align}\label{equality}
(\tau-1)z=0\mod\Lambda_{(1, \tau)}.
\end{align}
Since $(\tau-1)=\tau^{2}$, we conclude that
\begin{align}\notag
C_{1}\cap C_{\tau}=\{(0, 0)\}.
\end{align}
We claim that this is the only configuration that does the job.
First, let us try the configuration given by
\begin{align}\notag
w=0,\quad z=0,\quad w=z, \quad w=\tau^{2} z.
\end{align}
Observe that the curves $C_{1}$ and $C_{\tau^{2}}$, not only meet at
the origin, but also in other two distinct points. These points are
the two distinct zeros of the Weierstrass $\wp$-function associated
to the lattice $\Lambda_{(1, \tau)}$. More precisely, we have
\begin{align}\notag
C_{1}\cap C_{\tau^{2}}=\{(0, 0),\quad ((1-\tau^{2})/3,
(1-\tau^{2})/3),\quad ((\tau^{2}-1)/3, (\tau^{2}-1)/3)\}.
\end{align}
As the reader can easily verify, any other configuration can be
reduced to the two above or to the configuration
\begin{align}\notag
w=0,\quad z=0,\quad w=z, \quad w=-z,
\end{align}
which we already know not to be good. In conclusion, we have the
following lemma.

\begin{lemma}\label{final}
Let $(X, D)$ be a toroidal compactification with
$\overline{c}_{2}=1$ and $\kappa(X)=0$ for which $X$ is not birational to a bi-elliptic surface. Then $X$ is the blow up of
an Abelian surface $Y=\cc^{2}/\Gamma$ with $\Gamma=\Lambda_{(1,
\tau)}\times\Lambda_{(1, \tau)}$ and $\tau=e^{\frac{\pi i}{3}}$.
Moreover, the blow down divisor $C$ of $D$ is given by
\begin{align}\notag
w=0,\quad z=0,\quad w=z, \quad w=\tau z,
\end{align}
where $(w, z)$ are the natural product coordinates on $Y$.
\end{lemma}

\section{Proof of the main theorem and conclusions}\label{conclusion}

In this final section, we combine the results proved in Sections
\ref{different} and \ref{zero} to give a proof of Theorem
\ref{primo} stated in Section \ref{prelim}.

\begin{proof}[Proof of Theorem \ref{primo}]
Let $(X, D)$ be a toroidal compactification with
$\overline{c}_{2}=1$ for which $X$ is not birational to a bi-elliptic surface. Because of Proposition \ref{reduction}, we
know that $X$ is the blow up at just one point of a minimal surface
$Y$ of zero Kodaira dimension. Moreover, by Lemma \ref{abelian} we
know that $Y$ is an Abelian surface. Then, Lemma \ref{final} shows
that $Y$ has to be the product of two identical elliptic curves
associated to the lattice $\Lambda_{(1, \tau)}$, where
$\tau=e^{\frac{\pi i}{3}}$. We can then appeal to the results
contained in \cite{Holzapfel} to conclude that the pair $(X, D)$ is
indeed the compactification of a cusped complex hyperbolic surface.
Alternatively, one can observe that, by construction, the pair $(X,
D)$ given in Lemma \ref{final} saturates the logarithmic
Bogomolov-Miyaoka-Yau inequality. Then, because of the analytical
results contained in \cite{TianY}, it must be the compactification
of a ball quotient.
\end{proof}

It is worth mentioning that the pair $(Y, C)$ given in Lemma
\ref{final} is the starting point for an interesting construction of
Hirzebruch, see \cite{Hirzebruch}. More precisely, the pair $(X, D)$
and other blow ups of $(Y, C)$ are used by Hirzebruch as bases for a
clever branched cover construction which produces an infinite
sequence of minimal surfaces of general type whose ratio
$c^{2}_{1}/c_{2}$ tends to three. In the same paper, Hirzebruch
conjectured many of the surfaces he constructed to be
compactifications of Picard modular surfaces. This is indeed the
case as shown by Holzapfel \cite{Holzapfel}. For the particular
surface given in Lemma \ref{final}, one can also refer to
\cite{Stover}. In conclusion, the proof of Theorem \ref{secondo}
stated in Section \ref{prelim} not only depends on Theorem
\ref{primo}, but crucially relies on the fact that the pair given in
Lemma \ref{final} was already known to be the compactification of a
Picard modular surface \cite{Holzapfel}, \cite{Stover}.

\end{document}